\documentclass[]{article}
\usepackage{amsmath}
\usepackage{amsfonts}
\usepackage{graphicx}
\usepackage{amsthm}
\usepackage{mathtools}
\usepackage{hyperref}
\usepackage{amssymb}
\usepackage{xcolor}

\theoremstyle{plain}

\newtheorem{thm}{Theorem}

\newtheorem*{thm*}{Theorem}
\newtheorem{cor}{Corollary}[section]

\newtheorem{rem}{Remark}[section]
\title{A note on maximal subgroups of countable groups}
\author{Azer Akhmedov\footnote{Azer Akhmedov, Department of Mathematics,
North Dakota State University,
Fargo, ND, 58102, USA. E-mail: azer.akhmedov@ndsu.edu}}
\date{September 19, 2026}

\begin{document}
	\maketitle
	\begin{abstract}
		We show that for any countable group, the set of maximal subgroups is either at most countable or has cardinality $2^{\aleph _0}$. This answers a question from \cite{GM}.
      
	\end{abstract}

 The Margulis-Soifer theorem is an important dichotomy for linear groups. It states that a finitely generated linear group is either virtually solvable or contains a maximal subgroup of infinite index, and the two cases are mutually exclusive \cite{MS1, MS2, MS3}. Beyond linear groups, this dichotomy has been established for numerous other classes of groups, sometimes with minor modifications; we refer the reader to the beautiful survey \cite{GGS}.
 
 \medskip 

 In the proof of the theorem (see the proof of Theorem 4, \cite{MS3}), there is an interesting moment: for finitely generated non-virtually solvable linear groups, uncountably many infinite-index maximal subgroups are produced (in a finitely generated group, all but at most countably many of maximal subgroups will have infinite index). This raises the interesting question of whether one can claim that we indeed have exactly $2^{\aleph _0}$ maximal subgroups. It is in fact intriguing if the Continuum Hypothesis is lurking behind being fundamentally embedded in the nature of the question or can one avoid it to get a continuum family of maximal subgroups. One reasonable way to achieve the latter is to somehow produce a family explicit enough that it is parametrized by real numbers or by another set of cardinality $2^{\aleph _0}$. In \cite{GM}, this question has been studied  and (along with other remarkable results there) it is shown that $SL(n,\mathbb{Z}), n\geq 3$ has $2^{\aleph _0}$ maximal subgroups. The question for an arbitrary linear group was open (See Question 7.2 in \cite{GM}). We prove the following theorem.

  \medskip

\begin{thm} \label{thm:one}  Let $G$ be a countable group. Then the set of maximal subgroups of $G$ is either at most countable or has cardinality $2^{\aleph _0}$.  Similarly, the set of maximal subgroups of infinite index of $G$ is either at most countable or has cardinality $2^{\aleph _0}$.
\end{thm}

{\bf Proof. } Let $G$ be a countable group and $\mathrm{Sub}(G) = \{H : H\leq G\}$ be the set of subgroups of $G$. We will consider the Chabauty topology on $\mathrm{Sub}(G)$. The quickest way to define this topology is to consider the product topology on $\{0,1\}^{G}$ and identify every subset $S\subseteq G$ with its characteristic function $1_{S}$. Then $\mathrm{Sub}(G)$ is a subset of the topological space $2^G$ and the Chabauty topology is the subspace topology on $\mathrm{Sub}(G)$. For a subgroup $H\leq G$, basic open neighborhoods can be defined as $$U(H; K_{+}, K_{-}) = \{F\leq G: K_{+}\subseteq F, K_{-}\cap F = \emptyset \}$$ where $K_{+}, K_{-}$ are finite subsets of $G$, $K_{+}\subseteq H, K_{-}\subseteq G\backslash H$. Notice that $\mathrm{Sub}(G)$ is a closed subset of the compact metrizable space $\{0,1\}^{G}$, and hence is compact metrizable and Polish.

\medskip 

 For $g\in G$ define $U_g = \{H\in \mathrm{Sub}(G) : g\in H\}$.  Also, for $g, x \in G$, let $V_{g,x} = \{H : x\in \langle H, g\rangle\}$. Notice that $U_g$ is clopen for all $g\in G$. On the other hand, $V_{g,x}$ is open for all $g, x\in G$. To see this, let $H\in \mathrm{Sub}(G)$. For all $h_1, \dots , h_r\in H$, define $U_{h_1, \dots , h_r} = \{H: h_1, \dots , h_r\in H\}$. The set $U_{h_1, \dots , h_r}$ is clopen as an intersection of finitely many clopen subsets. Then we can write $$V_{g,x} = \displaystyle \mathop{\bigcup }_{\substack{h_1, \dots , h_r \\ w(g, h_1, \dots , h_r) = x}}U_{h_1, \dots , h_r}$$ and observe that $V_{g,x}$ is open as a union of open subsets.

\medskip 
  
 Then, since $G$ is countable, $\{H : \langle H, g \rangle = G\} = \displaystyle \mathop{\cap }_{x\in G}V_{g,x}$ is a $G_{\delta }$-set.

 \medskip 

 Let $\mathrm{Max}(G)$ be the set of all maximal subgroups of $G$ viewed as a subspace of $\mathrm{Sub}(G)$, and similarly, $\mathrm{Max}_{\infty }(G)$ be the set of all infinite index maximal subgroups of $G$ viewed as a subspace of $\mathrm{Sub}(G)$. We have $$\mathrm{Max}(G) = (\mathrm{Sub}(G)\backslash \{G\})\cap \displaystyle \mathop{\bigcap }_{g\in G}(U_g\cup \displaystyle \mathop{\bigcap }_{x\in G}V_{g,x}).$$

  Each term of the right-hand side  is a $G_{\delta }$-set, so $\mathrm{Max}(G)$ is a $G_{\delta }$-subset of $\mathrm{Sub}(G)$. Hence $\mathrm{Max}(G)$ is a Polish space.

  \medskip 

   For each $n\geq 0$, $\{H : [G:H]\geq n\}$ is open. Indeed, if $H$ has at least $n$ cosets, choose the representatives $g_1, \dots , g_n$ with $g_i^{-1}g_j\notin H$ for all $i\neq j$. These finitely many non-membership conditions persist in a neighborhood of $H$.  Then $$\{H : [G:H]= \infty \} = \displaystyle \mathop{\bigcap}_{n=1}^{\infty }\{H : [G:H]\geq n\}$$ is $G_{\delta }$. It follows that $$\mathrm{Max}_{\infty }(G) = \mathrm{Max}(G)\cap \{H : [G:H]= \infty \}$$ is also $G_{\delta }$ hence a Polish space. But every uncountable Polish space contains a perfect set and, therefore, has cardinality exactly $2^{\aleph _0}$ \cite{K}.  \ $\square $

   \medskip 

   \begin{rem} Notice that it is possible for a countable group to have uncountably many maximal subgroups, but at most countably many infinite-index maximal subgroups. A good example is the group $G = \displaystyle \bigoplus_{n=1}^{\infty }\mathop{\mathbb{Z}/2\mathbb{Z}}$. For this group, we have $|\mathrm{Max}(G)| = 2^{\aleph _0}$ but $|\mathrm{Max}_{\infty }(G)| = 0$.

   \end{rem}

    Combining Theorem \ref{thm:one} with Theorem 4 of \cite{MS3} we obtain the following corollary.  

\medskip 

   \begin{cor} Every finitely generated linear group is either virtually solvable or contains exactly $2^{\aleph _0}$ infinite index maximal subgroups.  
 \end{cor}

   \medskip

    {\em Acknowledgement:} I would like to thank Michael Cohen; I was his postdoctoral mentor at North Dakota State University, and during his time at NDSU, he taught me some descriptive set theory and Polish space techniques. I also thank ChatGPT for polishing my Polish space argument.

\end{document}